\documentclass[12pt]{article}

\usepackage{a4}   
\usepackage{amssymb,amsmath,mathrsfs,textcomp,pzccal}   
\usepackage[amsmath,thmmarks]{ntheorem}   
\usepackage{yhmath}   
\usepackage{indentfirst}   
\usepackage{tocloft}   
\usepackage[hyperfootnotes=false,colorlinks=true,linkcolor=black,citecolor=black]{hyperref}   

\usepackage{graphicx}
\usepackage{epstopdf}

\theoremstyle{nonumberplain}
\theorembodyfont{\rm}
\theoremseparator{.}

\theoremstyle{plain}
\theorembodyfont{\rm}
\theoremseparator{.}

\theoremstyle{plain}
\theorembodyfont{\slshape}
\theoremseparator{.}
\newtheorem{lemma}{Lemma}[section]
\newtheorem{corollary}{Corollary}[section]
\newtheorem{proposition}{Proposition}[section]
\newtheorem{definition}{Definition}[section]
\newtheorem{theorem}{Theorem}[section]

\theoremstyle{plain}
\theorembodyfont{\slshape}
\theoremseparator{$^\prime$.}

\theoremstyle{nonumberplain}
\theorembodyfont{\slshape}
\theoremseparator{.}

\theoremstyle{nonumberplain}
\theoremheaderfont{\bf}
\theorembodyfont{\rm} 
\theoremseparator{.}
\theoremsymbol{$\Box$}
\newtheorem{proof}{Proof}

\makeatletter
\renewcommand*{\@seccntformat}[1]{\csname the#1\endcsname.\quad}
\makeatother

\newcommand{\N}{\mathbb{N}}

\newcommand{\R}{\mathbb{R}}
\newcommand{\C}{\mathbb{C}}

\newcommand{\coloneqq}{\mathrel{\mathop:}=}

\providecommand{\abs}[1]{\lvert#1\rvert}

\providecommand{\norm}[1]{\lVert#1\rVert}

\renewcommand{\Re}{\text{Re}\,}

\providecommand{\chiup}{\raisebox{0.3ex}{$\chi$}}

\makeatletter
\def\moverlay{\mathpalette\mov@rlay}
\def\mov@rlay#1#2{\leavevmode\vtop{%
   \baselineskip\z@skip \lineskiplimit-\maxdimen
   \ialign{\hfil$\m@th#1##$\hfil\cr#2\crcr}}}
\newcommand{\charfusion}[3][\mathord]{
    #1{\ifx#1\mathop\vphantom{#2}\fi
        \mathpalette\mov@rlay{#2\cr#3}
      }
    \ifx#1\mathop\expandafter\displaylimits\fi}
\makeatother

\newcommand{\bigcupdot}{\charfusion[\mathop]{\bigcup}{\cdot}}

\begin{document}

%
%
\title{On compacts possessing strictly plurisubharmonic functions}
\author{Nikolay Shcherbina}
\maketitle
\nopagebreak

\vspace{0,2cm}
{\centerline {\small {\it Dedicated to the memory of my teacher Anatoli Georgievich Vitushkin (1931 - 2004).}}}

%
%

\vspace{0,5cm}
\small\noindent{\bf Abstract.}
We give a geometric condition on a compact subset of a complex manifold which is necessary and sufficient for the existence of a smooth strictly plurisubharmonic function defined in a neighbourhood of this set.

%
%
\renewcommand{\thefootnote}{}

%
%

%
%
%
%

%
%
%
%
%
\section{Introduction}

Plurisubharmonic functions play a central role in complex analysis. Many important and classical results are formulated in terms of these functions, in particular, using the existence of strictly plurisubharmonic functions on a given manifold. For example, Grauert \cite{Grauert58} characterized Stein manifolds by existence of smooth strictly plurisubharmonic exhaustion functions. This result was generalized to the case of complex spaces by Narasimhan in \cite{Narasimhan61} and \cite{Narasimhan62}. Sibony in \cite[Theorem 3, p. 362]{Sibony81} proved that the existence of a bounded smooth strictly plurisubharmonic function is sufficient for Kobayashi hyperbolicity of a complex manifold. A similar criterion for the existence of Bergman metric on Stein manifolds was established by Chen-Zhang in \cite[Theorem 1, p. 2998, and observation 2, p. 3002]{ChenZhang02}. Recently Poletsky \cite{Poletsky20} used manifolds possessing bounded smooth strictly plurisubharmonic functions to develop further the theory of pluricomplex Green functions.

In the present paper we study a question of the existence of smooth strictly plurisubharmonic functions on a given compact set. Smoothness and strict plurisubharmonicity of such functions can be defined as follows.

\begin{definition} \label{def 0} Let $\mathcal{K}$ be a compact subset of a complex manifold $\mathcal{M}$. We say that a function $\phi$ defined on $\mathcal{K}$ is smooth and strictly plurisubharmonic if there is a neighbourhood $\mathfrak{A}$ of $\mathcal{K}$ in $\mathcal{M}$ and a smooth strictly plurisubharmonic function $\varphi$ on $\mathfrak{A}$ such that $\varphi_{}|_{\mathcal{K}} = \phi$.
\end{definition}

The next result gives a complete geometric characterization of compacts possessing such functions.

\vspace{0,3cm}
\noindent \textbf{Main Theorem.} {\it Let $\mathcal{K}$ be a compact subset of a complex manifold. Then $\mathcal{K}$ possesses a smooth strictly plurisubharmonic function if and only if $\mathcal{K}$ does not have 1-pseudoconcave subsets.}

1-pseudoconcavity here is understood in the sense of Rothstein {\rm \cite{Rothstein55}}. By a complex manifold we will always mean a manifold of pure complex dimension which has a Hausdorff topology with a countable basis.

The paper is organized as follows. In Section 2 we recall the basic definitions and the main properties of pseudoconcave sets as well as the construction of a special plurisubharmonic function given in \cite[Theorem 3.1, part 1]{HarzShcherbinaTomassini17}. In Section 3 we provide a constructive way to define the maximal $1$-pseudoconcave subset of a given compact set. In Section 4 we prove the Main Theorem and one of its corollaries. Finally, in Section 5 we give some applications of our results and discuss their relation to the other topics.

\vspace{0,2cm}
\noindent 
{\bf Remark.} When our paper was already posted on arXiv and submitted to a journal, Nessim Sibony
informed us that an alternative proof of the `if'
part of our Main Theorem can be derived from \cite{Sibony2018} and \cite{FornaessSibony95}.
Namely, if there are no strictly psh functions on a compact set, then the duality
argument in the proof of \cite[Proposition 2.1]{Sibony2018} can be extended to produce a
non-trivial $dd^c$-closed positive current of bidimension $(1,1)$ supported in that
compact set. (A similar duality argument was earlier used in \cite[Theorem (38)]{HarveyLawson83}.)
The support of such a current is a $1$-pseudoconvex set by \cite[Corollary 2.6]{FornaessSibony95}.
We are grateful to Nessim Sibony for this valuable observation but believe that our
direct geometric proof is of independent interest.

\section{Preliminaries} \label{sec_preliminaries}

We recall first the notion of {\em $1$-pseudoconvexity} in the sense of Rothstein. Let $\Delta^n \coloneqq \{z \in \C^n : \norm{z}_\infty < 1 \}$, where $\norm{z}_\infty = \max_{1 \le j \le n} \abs{z_j}$. An $(1, n - 1)$ Hartogs figure $H$ is a set of the form
\[H = \big\{(z_1, \cdots, z_n) \in \Delta^1 \times \Delta^{n-1} :  |z_1| < r_1 \text{ or } \norm{(z_2, \cdots, z_n)}_\infty > r_2 \big\}, \]
where $0 < r_1,r_2 < 1$, and we write $\hat{H} \coloneqq \Delta^n$. 

\begin{definition} \label{def 1} Let $\mathcal{M}$ be a complex manifold of dimension $n$. An open set $\Omega \subset \mathcal{M}$ is called  $1$-pseudoconvex in $\mathcal{M}$ if it satisfies the Kontinuit\"atssatz with respect to $(n-1)$-polydiscs in $\mathcal{M}$, i.e., if for every $(1, n - 1)$ Hartogs figure and every injective holomorphic mapping $\Phi \colon \hat{H} \to \mathcal{M}$ such that $\Phi(H) \subset \Omega$ one has $\Phi(\hat{H}) \subset \Omega$.
\end{definition}

This definition was introduced by Rothstein \cite{Rothstein55} in a more general setting of $q$-pseudoconvex sets for every $q = 1, 2, \ldots, n-1$.  We restrict our definition to the special case $q = 1$, since in the present paper we only need the notion of $1$-pseudoconvexity. 

Another way to define $1$-pseudoconvexity can be described as follows. For an arbitrary $r \in (0 ,1)$ we consider a {\em spherical hat}

\[ {\mathbb S}^n_r := \{z = (z_1, \ldots, z_n) \in {\mathbb C}^n : \|z \|^2 = 1, \,\, x_1 := \Re z_1 \geq r\} \]

\noindent
and a {\em filled spherical hat}

\[ {\hat {\mathbb S}}^n_r := \{z = (z_1, \ldots, z_n) \in {\mathbb C}^n : \|z \|^2 \leq 1, \,\, x_1 := \Re z_1 \geq r\}. \]

\begin{definition} \label{def 2}  Let $\mathcal{M}$ be a complex manifold of dimension $n$. An open set $\Omega \subset \mathcal{M}$ is called  $1$-pseudoconvex in $\mathcal{M}$ if for every $r \in (0 ,1)$, every neighbourhood $U := U({\hat {\mathbb S}}^n_r) \subset {\mathbb C}^n$ of the filled spherical hat ${\hat {\mathbb S}}^n_r$ and every injective holomorphic mapping $\Phi \colon U \to \mathcal{M}$ such that $\Phi({\mathbb S}^n_r) \subset \Omega$ one has $\Phi({\hat {\mathbb S}}^n_r) \subset \Omega$. 
\end{definition}

The next statement shows that the above definitions give us the same notion.

\begin{proposition} \label{thm_lmp}
	Let $\mathcal{M}$ be a complex manifold and $\Omega \subset \mathcal{M}$ be an open set. Then the following assertions are equivalent:
	\begin{enumerate}
		\item[$(1)$] $\Omega$ is $1$-pseudoconvex in the sense of Definition \ref{def 1}.
		\item[$(2)$] $\Omega$ is $1$-pseudoconvex in the sense of Definition \ref{def 2}.
	\end{enumerate}
	\end{proposition}

\begin{proof}
	To prove the implication $(1) \Rightarrow (2)$ we argue by contradiction and assume that there is a domain $\Omega$ which is $1$-pseudoconvex in the sense of Definition \ref{def 1}, but not $1$-pseudoconvex in the sense of Definition \ref{def 2}. Then, in view of  Definition \ref{def 2}, and after the substitution of $r$ by $r - \varepsilon$ with $\varepsilon > 0$ small enough if necessary, we can assume that $\Phi({\mathbb S}^n_r) \subset \Omega$, but $\Phi{\big(}{\rm Int}({\hat {\mathbb S}}^n_r){\big)} \cap ({\mathcal{M}} \setminus \Omega) \neq \emptyset$, where by ${\rm Int}({\hat {\mathbb S}}^n_r)$ we will denote the interior of the set ${\hat {\mathbb S}}^n_r$. Consider now for each $C \geq 0$ a slightly more general spherical hat

\[ {\mathbb S}^n_{r, C} := \{z = (z_1, \ldots, z_n) \in {\mathbb C}^n : |z_1 + C|^2 + |z_2|^2 + \cdots + |z_n|^2 = C^2 +2rC + 1, \,\, x_1 \geq  r\} \]

\noindent
and a corresponding filled spherical hat

\[ {\hat {\mathbb S}}^n_{r, C} := \{z = (z_1, \ldots, z_n) \in {\mathbb C}^n : |z_1 + C|^2 + |z_2|^2 + \cdots + |z_n|^2 \leq C^2 +2rC + 1, \,\, x_1 \geq  r\} \]

\noindent
and observe that the spherical hats ${\mathbb S}^n_{r, C}$ depend continuously on the parameter $C$, they all contain the "boundary" 
\[ {\partial}{\mathbb S}^n_r := \{z = (z_1, \ldots, z_n) \in {\mathbb C}^n : \|z \|^2 = 1, \,\, x_1 = r\} \]
 of the spherical hat ${\mathbb S}^n_r$ and, moreover, that $\bigcap_{C > 0} {\rm Int}({\hat {\mathbb S}}^n_{r, C})  = \emptyset$. Hence, there is 
\[ C_0 := \min\{C : \Phi{\big(}{\mathbb S}^n_{r, C}{\big)} \cap ({\mathcal{M}} \setminus \Omega) \neq \emptyset \} > 0. \]

\noindent
If $z_0 \in {\mathbb S}^n_{r, C_0}$ is a point such that $\Phi(z_0) \in {\mathcal{M}} \setminus \Omega$, then we can consider an affine transformation $\mathbb L$ of ${\mathbb C}^n$ which sends the sphere ${\mathbb S}^n := \{z = (z_1, \ldots, z_n) \in {\mathbb C}^n : \|z \|^2 = 1\}$ to the sphere ${\mathbb S}^n_{C_0} := \{z = (z_1, \ldots, z_n) \in {\mathbb C}^n : |z_1 + C_0|^2 + |z_2|^2 + \cdots + |z_n|^2 = C^2_0 +2r{C_0} + 1\}$, the origin $O$ to the point $(- C_0, 0, \cdots, 0)$ and the point $(1, 0, \cdots, 0)$ to the point $z_0$. If now for $r' < 1$ close enough to $1$ we consider the spherical hat ${\mathbb S}^n_{r'}$, the neighbourhood ${\mathbb L}^{- 1}{\big(}U({\hat {\mathbb S}}^n_r){\big)}$ of the  filled spherical hat ${\hat {\mathbb S}}^n_{r'}$ and the injective holomorphic mapping $\Phi \circ \mathbb L \colon {\mathbb L}^{- 1}{\big(}U({\hat {\mathbb S}}^n_r){\big)} \to \mathcal{M}$, then, by construction, we will get that ${\big(}\Phi \circ \mathbb L{\big)}(1, 0, \cdots, 0) \in {\mathcal{M}} \setminus \Omega$ and for $\delta > 0$ small enough we will also get that ${\big(}\Phi\circ \mathbb L{\big)}{\big(}{U_{2\delta}({\mathbb S}}^n_{r'}) \setminus {\overline {\mathbb B}}^n_1(0){\big)} \subset \Omega$, where ${U_{2\delta}({\mathbb S}}^n_{r'})$ is the $2\delta$-neighbourhood of the set ${\mathbb S}^n_{r'}$ in ${\mathbb C}^n$ and ${\mathbb B}^n_1(0)$ is the ball  in ${\mathbb C}^n$ of radius $1$ with center at the origin. Hence, if for ${\delta}' > \delta$ close enough to $\delta$, $r_1 \in (0 ,{\delta}')$ close enough to $0$ and $r_2 \in (0 ,\delta)$ close enough to $\delta$ we consider the Hartogs figure

\[H_{{\delta}',\delta,r_1 ,r_2} = \big\{(z_1, \ldots, z_n) \in \Delta^1_{{\delta}'}(1 + \delta) \times \Delta^{n-1}_{\delta}(0) :  |z_1 - (1 + \delta)| < r_1 \text{ or } \norm{(z_2, \cdots, z_n)}_\infty > r_2 \big\} \]

\noindent 
and write

\[{\hat H}_{{\delta}',\delta,r_1 ,r_2} :=  \Delta^1_{{\delta}'}(1 + \delta) \times \Delta^{n-1}_{\delta}(0), \]

\noindent
where $\Delta_s(a) := \{z \in {\mathbb C} : |z - a| < s\}$, then we will get that $H_{{\delta}',\delta,r_1 ,r_2} \subset {\mathbb C}^n \setminus {\overline {\mathbb B}}^n_1(0)$, but $(1 , 0 , \cdots , 0) \in {\hat H}_{{\delta}',\delta,r_1 ,r_2} \cap {\overline {\mathbb B}}^n_1(0)$. If we now define an affine  change of coordinates ${\mathbb L}'$ in $\mathbb{C}^n$ by 

\[z_1 \to {\delta}' z_1 + (1 + \delta) =: z'_1, \, z_j \to \delta z_j  =: z'_j \text{ for } j = 2, 3, \cdots , n, \]

\noindent
then the map $\Phi \circ {\mathbb L} \circ {\mathbb L}' : \hat H \to \mathcal M$ will give us the desired contradiction to the assumption on $1$-pseudoconvexity of the domain $\Omega$ in the sense of Definition \ref{def 1}.

To prove the implication $(2) \Rightarrow (1)$ we will follow the argument used in the proof of Theorem 3.2 in \cite{HarzShcherbinaTomassini17} and assume, to get a contradiction, that there is a domain $\Omega$ which is $1$-pseudoconvex in the sense of Definition \ref{def 2}, but not $1$-pseudoconvex in the sense of Definition \ref{def 1}. Then for some $0 < r_1,r_2 < 1$ there exists a $(1, n - 1)$ Hartogs figure $H = \big\{(z_1 , \cdots ,z_n) \in \Delta \times \Delta^{n-1} :  |z_1| < r_1 \text{ or } \norm{(z_2, \cdots, z_n)}_\infty > r_2 \big\}$ and an injective holomorphic mapping $\Phi \colon \hat{H} \to \mathcal{M}$ such that $\Phi(H) \subset \Omega$ but $\Phi(\hat{H}) \cap (\mathcal{M} \setminus \Omega) \neq \emptyset$. For small $\varepsilon > 0$, let $\varphi \colon \C^\ast_{z_1} \times \C^{n-1}_{(z_2, \cdots, z_n)} \to \R$ be the smooth strictly plurisubharmonic function defined by $\varphi(z) \coloneqq -\log\abs{z_1} + \varepsilon \norm{z}^2 $, and for each $C \in \R$ let $G_C$ denote the domain $G_C \coloneqq \big\{\zeta \in \Phi(\hat{H}) : (\varphi \circ \Phi^{-1})(\zeta) < C\big\}$. Since for $C$ large enough the set $\hat{H} \cap \{z \in \C^n : \varphi < C\}$ contains $\hat{H} \setminus H$, and since $\Phi(\hat{H}) \cap (\mathcal{M} \setminus \Omega) \subset \Phi(\hat{H} \setminus H)$, we know that for $C$ large enough $\Phi(\hat{H}) \cap (\mathcal{M} \setminus \Omega) \subset G_C$. Let $C_0 \coloneqq \inf \{C \in \R : \Phi(\hat{H}) \cap (\mathcal{M} \setminus \Omega) \subset G_C\}$. Then the set $\mathcal{M} \setminus \Omega$ "touches" the strictly pseudoconvex part $M \coloneqq bG_{C_0} \cap \Phi(\hat{H})$ of the boundary $bG_{C_0}$ of $G_{C_0}$ "from inside". That is $\Phi(\hat{H}) \cap (\mathcal{M} \setminus \Omega) \subset {\bar G}_{C_0}$, there is a point $\zeta_0 \in \Phi(\hat{H}) \cap (\mathcal{M} \setminus \Omega) \cap bG_{C_0}$ and the domain $G_{C_0}$ is strictly pseudoconvex near the point $\zeta_0$. It follows from strict pseudoconvexity of the domain $G_{C_0}$ near $\zeta_0$ that there is a bounded strictly convex domain $G_0 \subset \mathbb{C}^n$, a point $z_0 \in bG_0$, a neighbourhood $U_0$ of $z_0$ and an injective holomorphic mapping $\Psi : U_0 \to \mathcal M$ such that $\Psi {\big(}U_0 \cap G_0\big) = \Psi{\big(}U_0{\big)} \cap G_{C_0}$ and $\Psi(z_0) = \zeta_0$. Strict convexity of the domain $G_0 \subset \mathbb{C}^n$ implies that we can find a ball ${\mathbb B}^n_R(\tilde C)$  centered at a point $\tilde C \in \mathbb{C}^n$ with radius $R > 0$ which contains the domain $G_0$ and such that ${b{\mathbb B}}^n_R(\tilde C) \cap {\bar{G_0}} = z_0$. If we take an affine  change of coordinates $\mathbb L$ in $\mathbb{C}^n$ which sends the ball ${\mathbb B}^n_1(0)$ to the ball ${\mathbb B}^n_R(\tilde C)$, the origin $O$ to the point $\tilde C$ and the point $(1, 0, \cdots , 0)$ to the point $z_0$ and then consider a translation ${\mathbb E}_{\delta} := {\mathbb E} + \delta (1, 0, \cdots , 0)$, where $\mathbb E$ is the identity map of $\mathbb{C}^n$ and $\delta > 0$ is small enough, we will see from our construction that the injective holomorphic mapping $\Psi \circ \mathbb L \circ {\mathbb E}_{\delta}$ applied to a neighbourhood of the filled spherical hat ${\hat {\mathbb S}}^n_r$ with $0 < r < 1$ close enough to $1$ will give us a contradiction to the assumption of $1$-pseudoconvexity of the domain $\Omega$ in the sense of Definition \ref{def 2}.
\end{proof}

Now we recall the definition of {\em $1$-pseudoconcavity} for closed sets.

\begin{definition} \label{def 3} Let $\mathcal{M}$ be a complex manifold and $A \subset \mathcal{M}$ be a closed set. Then $A$ is called  $1$-pseudoconcave in $\mathcal{M}$ if $\mathcal{M} \setminus A$ is $1$-pseudoconvex in $\mathcal{M}$. 
\end{definition}

Note that more equivalent descriptions of $1$-pseudoconvex sets are known. For example, from Theorems 4.2 and 5.1 of S{\l}odkowski \cite{Slodkowski86} it follows, in particular, that a nonempty relatively closed subset $A$ of an open set $V \subset \C^n$ is $1$-pseudoconcave in $V$ if and only if plurisubharmonic functions have the local maximum property on $A$. We will need here an analogous statement in a more general setting of complex manifold.

\begin{proposition} \label{thm_lmp}
	Let $\mathcal{M}$ be a complex manifold and $A \subset \mathcal{M}$ be a closed set. Then the following assertions are equivalent:
	\begin{enumerate}
		\item[$(1)$] For every $\zeta \in A$, there exists a neighbourhood $V \subset \mathcal{M}$ of $\zeta$ such that $A \cap V$ is $1$-pseudoconcave in $V$.
		\item[$(2)$] A is $1$-pseudoconcave in $\mathcal{M}$.
		\item[$(3)$] For every $\zeta \in A$, there exists a neighbourhood $V \subset \mathcal{M}$ of $\zeta$ such that for every compact set $B \subset V$ and every plurisubharmonic function $\varphi$ defined in a neighbourhood of $B$ one has $\max_{A \cap B} \varphi \le \max_{A \cap bB} \varphi$.
	\end{enumerate}
	Here $\max_{A \cap bK} \varphi$ is meant to be $-\infty$ if $A \cap bK = \emptyset$.
\end{proposition}

A detailed proof of this statement (which follows the ideas of S{\l}odkowski \cite{Slodkowski86}) can be found in a more general setting of $q$-pseudoconcave sets in \cite[Proposition 3.3]{HarzShcherbinaTomassini14}.

The following result was proved in \cite[Theorem 3.1, part 1]{HarzShcherbinaTomassini17}. Since it plays an important role in this paper, we will
present it here in details for the reader convenience in a slightly different form adapted to the current presentation.

\begin{theorem} \label{positive_psh}
There exists a domain $W$ in $\C^n$ with coordinates $(z_1, z_2, \ldots, z_n)$, $z_j = x_j+iy_j$, and a smooth plurisubharmonic function $\varphi : W \to [0 ,+ \infty)$ such that
\begin{enumerate}
	\item[(1)] $\Pi_{-} := \{z \in \C^n : x_1 \leq 0\} \subset W$.
	\item[(2)] $\varphi = 0$ on $\Pi_{-}$.
	\item[(3)] $\varphi > 0$ on $W \setminus \Pi_{-}$.
	\item[(4)] $\varphi$ is strictly plurisubharmonic on $W \setminus \Pi_{-}$.
\end{enumerate}
\end{theorem}

\begin{proof}
	For every $j \in \N$, let $\psi_j \colon {\mathbb B}^n_j(0) \to \R$ be the smooth and strictly plurisubharmonic function defined by
	\[ \psi_j(z_1, \ldots, z_n) \coloneqq x_1 - \frac{1}{2^{j-2}} + \frac{1}{j^22^{j-1}}\big(y_1^2 + \abs{z_2}^2 + \cdots + \abs{z_n}^2\big). \]
	Choose a smooth function $\chiup_j \colon \R \to [0,\infty)$ such that $\chiup_j \equiv 0$ on $(-\infty, -1/2^j]$ and such that $\chiup_j$ is strictly increasing and strictly convex on $(-1/2^j, \infty)$. Set $\widetilde{\varphi}_j \coloneqq \chiup_j \circ \psi_j$. Then $\widetilde{\varphi}_j$ is a smooth plurisubharmonic function on ${\mathbb B}^n_j(0)$ such that $\widetilde{\varphi}_j \equiv 0$ on $\{\psi_j \le -1/2^j\} \supset {\mathbb B}^n_j(0) \cap \{x_1 \le 1/2^j\}$ and such that $\widetilde{\varphi}_j$ is strictly plurisubharmonic and positive on $\{\psi_j > -1/2^j\} \supset {\mathbb B}^n_j(0) \cap \{x_1 > 3/2^j\}$. Thus
	\[ \varphi_j(z) \coloneqq \left\{\begin{array}{c@{\,,\quad}l} \widetilde{\varphi}_j(z) & z \in {\mathbb B}^n_j(0) \cap \{x_1 \ge 1/2^j\} \\ 0 & z \in \{x_1 < 1/2^j\} \end{array} \right. \]
	is a smooth plurisubharmonic function on $W_j \coloneqq {\mathbb B}^n_j(0) \cup \{x_1 < 1/2^j\}$ such that $\varphi_j$ is strictly plurisubharmonic and positive on ${\mathbb B}^n_j(0) \cap \{x_1 > 3/2^j\}$. Observe that $W \coloneqq \bigcap_{j=1}^\infty W_j$ is a connected open neighbourhood of $\{x_1 \le 0\}$. Then one easily sees that for a sequence $\{\varepsilon_j\}_{j=1}^\infty$ of positive numbers that converges to zero fast enough, the function $\varphi \coloneqq \sum_{j=1}^\infty \varepsilon_j\varphi_j$ is smooth and plurisubharmonic on $W$ such that $\varphi \equiv 0$ on $\{x_1 \le 0\}$ and such that $\varphi$ is strictly plurisubharmonic and positive on $W \cap \{x_1 > 0\}$, which completes the proof of the theorem.
\end{proof}

As a consequence of Theorem \ref{positive_psh} we get the following statement which will be one of the main technical tools in our construction below.

\begin{corollary} \label{hat_psh}
	For every $r \in (0 ,1)$ there is a smooth nonnegative plurisubharmonic function $\varphi_r$ defined on the domain $\Omega_r := {\mathbb C}^n \setminus {\mathbb S}^n_r$ such that
	\begin{enumerate}
		\item[(1)] $\varphi_r$ is equal to $0$ on the set $\Omega_r \setminus {\hat {\mathbb S}}^n_r$.
		\item[(2)] $\varphi_r$ is positive and strictly plurisubharmonic in the interior ${\rm Int}({\hat {\mathbb S}}^n_r)$ of the set ${\hat {\mathbb S}}^n_r$.
	\end{enumerate}
\end{corollary}

\begin{proof}
	If for each $r \in (0 ,1)$ we define the function $\varphi_r$ as
	\[ \varphi_r(z) \coloneqq \left\{ \begin{array}{c@{\,,\quad}l} \varphi((z_1 - r, z_2, \cdots, z_n))  &  {\rm{for}} \,\,  z = (z_1 , z_2, \cdots, z_n) \in {\rm Int}({\hat {\mathbb S}}^n_r), \\ 0 &   {\rm{for}} \,\,  z \in \Omega_r \setminus {\hat {\mathbb S}}^n_r,   \end{array} \right.	\]
	where $\varphi$ is the function constructed in Theorem \ref{positive_psh}, then it is straightforward to see that the function $\varphi_r$ has all the desired properties. 
\end{proof}

\section{Construction and properties of the set $\mathfrak{n}(\mathcal{K})$} 

Let ${\mathcal K}' \supset {\mathcal K}''$ be compact sets in a complex manifold ${\mathcal M}$ of dimension n. We say that the set ${\mathcal K}''$ is obtained from the set ${\mathcal K}'$ by a {\em spherical cut}  if there exist $r \in (0 ,1)$, a neighbourhood $U := U({\hat {\mathbb S}}^n_r) \subset {\mathbb C}^n$ of the filled spherical hat ${\hat {\mathbb S}}^n_r$ and an injective holomorphic mapping $\Phi \colon U \to \mathcal{M}$ such that $\Phi({\mathbb S}^n_r) \subset \mathcal M \setminus {\mathcal K}'$ and ${\mathcal K}' \setminus \Phi({\rm Int}({\hat {\mathbb S}}^n_r)) = {\mathcal K}''$. 

Further, for a pair of compact sets ${\mathcal K}' \supset {\mathcal K}''$ in $\mathcal M$ we say that the set ${\mathcal K}''$ is obtained from the set ${\mathcal K}'$ by a {\em finite sequence of spherical cuts} if there exists a finite decreasing sequence ${\mathcal K}_1 \supset {\mathcal K}_2 \supset \cdots \supset {\mathcal K}_m$ of compact sets in ${\mathcal M}$ such that ${\mathcal K}_1 = {\mathcal K}'$, ${\mathcal K}_m = {\mathcal K}''$ and for each $j = 2, 3, \cdots, m$ the set ${\mathcal K}_j$ is obtained from the set ${\mathcal K}_{j - 1}$ by a spherical cut.

Then, for a given compact set ${\mathcal K}$ in ${\mathcal M}$, we can consider the family ${\mathcal F}_{\mathcal K}$ of compact subsets of ${\mathcal K}$ defined by

\[ {\mathcal F}_{\mathcal K} := \{{\mathcal K}_\alpha : {\mathcal K}_\alpha {\rm{\,\, is \,\, obtained \,\, from \,\, }} {\mathcal K} {\rm{\,\, by \,\, a \,\, finite \,\, sequence \,\, of \,\, spherical \,\, cuts}}\}_{\alpha \in {\mathcal A}}, \]
where ${\mathcal A}$ is a parameter set of this family.

The next statement follows easily from the definition of ${\mathcal F}_{\mathcal K}$.

\begin{lemma} \label{intersection}
	Let ${\mathcal K}_{{\alpha}_1}, {\mathcal K}_{{\alpha}_2}, \cdots , {\mathcal K}_{{\alpha}_m}$ be a finite set of compacts from ${\mathcal F}_{\mathcal K}$, then  one also has that $\bigcap^m_{j=1}{\mathcal K}_{{\alpha}_j} \in {\mathcal F}_{\mathcal K}$.
\end{lemma}

\begin{proof}
	Since the general case by induction can be reduced to the case of two sets ${\mathcal K}_{{\alpha}_1}$ and ${\mathcal K}_{{\alpha}_2}$, we will only treat this case here. 
	
	First, we observe that by the assumption ${\mathcal K}_{{\alpha}_1} \in {\mathcal F}_{\mathcal K}$ we get a finite decreasing sequence ${\mathcal K}'_1 \supset {\mathcal K}'_2 \supset \cdots \supset {\mathcal K}'_{m_1}$ of compact sets in ${\mathcal M}$ such that ${\mathcal K}'_1 = \mathcal K$, ${\mathcal K}'_{m_1} = {\mathcal K}_{{\alpha}_1}$ and for each $j = 2, 3, \cdots, m_1$ the set ${\mathcal K}'_j$ is obtained from the set ${\mathcal K}'_{j - 1}$ by a spherical cut. Similarly, the assumption ${\mathcal K}_{{\alpha}_2} \in {\mathcal F}_{\mathcal K}$ implies that there is a finite decreasing sequence ${\mathcal K}''_1 \supset {\mathcal K}''_2 \supset \cdots \supset {\mathcal K}''_{m_2}$ of compact sets in ${\mathcal M}$ such that ${\mathcal K}''_1 = \mathcal K$, ${\mathcal K}''_{m_2} = {\mathcal K}_{{\alpha}_2}$ and for each $j = 2, 3, \cdots, m_2$ the set ${\mathcal K}''_j$ is obtained from the set ${\mathcal K}''_{j - 1}$ by a spherical cut. If now we consider the spherical cuts corresponding to the first sequence with the initial set ${\mathcal K}$ and then the spherical cuts corresponding to the second sequence, but applied to the initial set ${\mathcal K}_{{\alpha}_1}$, we will get a finite decreasing sequence ${\mathcal K} = {\mathcal K}'_1 \supset {\mathcal K}''_2 \supset \cdots \supset {\mathcal K}'_{m_1} = {\mathcal K}_{{\alpha}_1} = {\mathcal K}_{{\alpha}_1} \cap {\mathcal K} = {\mathcal K}_{{\alpha}_1} \cap {\mathcal K}''_1 \supset {\mathcal K}_{{\alpha}_1} \cap {\mathcal K}''_2 \cdots \supset {\mathcal K}_{{\alpha}_1} \cap {\mathcal K}''_{m_2} = {\mathcal K}_{{\alpha}_1} \cap {\mathcal K}_{{\alpha}_2}$ of compact sets in ${\mathcal M}$ with the property that each set of this sequence is obtained from the previous set by a spherical cut. Since the initial set of this sequence is ${\mathcal K}$ and the final set is ${\mathcal K}_{{\alpha}_1} \cap {\mathcal K}_{{\alpha}_2}$, we conclude that ${\mathcal K}_{{\alpha}_1} \cap {\mathcal K}_{{\alpha}_2} \in {\mathcal F}_{\mathcal K}$. This completes the proof of Lemma \ref{intersection}.
\end{proof}

Now we can define the set $\mathfrak{n}(\mathcal{K})$ which plays a special role in the present article and will be called in what follows the {\it nucleus} of ${\mathcal K}$:
\[
\mathfrak{n}(\mathcal{K}) := \bigcap_{{\mathcal K}_{\alpha} \in {\mathcal F}_{\mathcal K}}{\mathcal K}_{\alpha}. 
\]

The most important for us properties of this set are given in the next statement.

\begin{theorem} \label{K-star}
The set $\mathfrak{n}(\mathcal{K})$ is $1$-pseudoconcave. Moreover, $\mathfrak{n}(\mathcal{K})$ is the maximal $1$-pseudoconcave subset of the set $\mathcal K$.
\end{theorem}

\begin{proof}
     Assume, to get a contradiction, that the set $\mathfrak{n}(\mathcal{K})$ is not $1$-pseudoconcave. Then, in view of Definition \ref{def 2} and Proposition \ref{thm_lmp}, for some $r \in (0 ,1)$ there exist a neighbourhood $U \subset {\mathbb C}^n$ of the filled spherical hat ${\hat {\mathbb S}}^n_r$ and an injective holomorphic mapping $\Phi \colon U \to \mathcal{M}$ such that $\Phi({\mathbb S}^n_r) \subset \mathcal{M} \setminus \mathfrak{n}(\mathcal{K})$ and $\Phi({\hat {\mathbb S}}^n_r) \cap \mathfrak{n}(\mathcal{K}) \neq \emptyset$. Note that, after the substitution of $r$ by $r - \varepsilon$ with $\varepsilon > 0$ small enough if necessary, but keeping the same $U$ and $\Phi$, we can achieve that $\Phi({\mathbb S}^n_r) \subset \mathcal{M} \setminus \mathfrak{n}(\mathcal{K})$ and $\Phi{\big(}{\rm Int}({\hat {\mathbb S}}^n_r){\big)} \cap \mathfrak{n}(\mathcal{K}) \neq \emptyset$. Let $\zeta$ be a point of the set $\Phi{\big(}{\rm Int}({\hat {\mathbb S}}^n_r){\big)} \cap \mathfrak{n}(\mathcal{K})$. Since $\Phi({\mathbb S}^n_r) \subset \mathcal{M} \setminus \mathfrak{n}(\mathcal{K})$, then, in view of compactness of the sets $\Phi({\mathbb S}^n_r)$ and $\mathfrak{n}(\mathcal{K})$, there is a neighbourhood $V$ of $\mathfrak{n}(\mathcal{K})$ such that $\Phi({\mathbb S}^n_r) \cap V = \emptyset$. It follows now from the definition of the set $\mathfrak{n}(\mathcal{K})$ that there is a finite set of compacts ${\mathcal K}_{{\alpha}_1}, {\mathcal K}_{{\alpha}_2}, \cdots , {\mathcal K}_{{\alpha}_m}$ in the family ${\mathcal F}_{\mathcal K}$ such that $\bigcap^m_{j=1}{\mathcal K}_{{\alpha}_j} \subset V$. By Lemma \ref{intersection} we know that the set $\overset{\thicksim}{\mathcal K} := \bigcap^m_{j=1}{\mathcal K}_{{\alpha}_j}$ also belongs to the family ${\mathcal F}_{\mathcal K}$. Since $\overset{\thicksim}{\mathcal K} \subset V$ and $\Phi({\mathbb S}^n_r) \cap V = \emptyset$, we see that $\Phi({\mathbb S}^n_r) \cap \overset{\thicksim}{\mathcal K} = \emptyset$ and, hence, we can make one more spherical cut to obtain the set $\overset{\thicksim}{\mathcal K} \setminus \Phi({\rm Int}({\hat {\mathbb S}}^n_r)) =: \overset{\thickapprox}{\mathcal{K}}$ from the set $\overset{\thicksim}{\mathcal K}$. Observe now that, by construction, the set $\overset{\thickapprox}{\mathcal{K}}$ also belongs to the family ${\mathcal F}_{\mathcal K}$ and, moreover, that $\zeta \notin \overset{\thickapprox}{\mathcal{K}}$. Then, by the definition of the set $\mathfrak{n}(\mathcal{K})$, we know that $\zeta \notin \mathfrak{n}(\mathcal{K})$ which contadicts our choice of the point $\zeta$.

     To prove the maximality of $\mathfrak{n}(\mathcal{K})$ we observe first that, if ${\mathcal K}_2 \subset {\mathcal K}_1$ are compact sets in ${\mathcal M}$ such that ${\mathcal K}_2$ is obtained from ${\mathcal K}_1$ by a spherical cut, and if ${\mathcal K}' \subset {\mathcal K}_1$ is a compact set which is $1$-pseudoconcave, then, by the definition of $1$-pseudoconcave sets, we also have that ${\mathcal K}' \subset {\mathcal K}_2$. Applying this argument to a 
     finite sequence of spherical cuts of the set $\mathcal K$, we see that for an arbitrary $1$-pseudoconcave compact subset ${\mathcal K}'$ of $\mathcal K$ the inclusion ${\mathcal K}' \subset {\mathcal K}_\alpha$ holds true for every ${\mathcal K}_{\alpha} \in {\mathcal F}_{\mathcal K}$. Thus, by the definition of $\mathfrak{n}(\mathcal{K})$, we have that ${\mathcal K}' \subset \bigcap_{{\mathcal K}_{\alpha} \in {\mathcal F}_{\mathcal K}}{\mathcal K}_{\alpha} = \mathfrak{n}(\mathcal{K})$. This completes the proof of Theorem \ref{K-star}.
\end{proof}

\section{Proof of the Main Theorem}

Now we can complete the proof of the Main Theorem. In order to do this we distinguish two cases.
       
\vspace{2mm}
\noindent{{\bf Case 1.} \sl $\mathfrak{n}(\mathcal{K}) \neq \emptyset$.}

\vspace{3mm} In this case we prove that the set $\mathcal K$ does not possess smooth strictly plurisubharmonic functions. Indeed, we argue by contradiction and assume that, in accordance with Definition \ref{def 0}, there is a neighbourhood $\mathfrak{A}$ of $\mathcal{K}$ in $\mathcal{M}$ and a smooth strictly plurisubharmonic function $\varphi$ on $\mathfrak{A}$. Since the set $\mathfrak{n}(\mathcal{K})$ is compact, there is a point $\zeta_0 \in \mathfrak{n}(\mathcal{K})$ such that $\varphi(\zeta_0) = \max_{\zeta \in \mathfrak{n}(\mathcal{K})}\varphi(\zeta)$. Then, in view of strict plurisubharmonicity of $\varphi$, in local coordinates near $\zeta_0$ the function $\varphi(\zeta_0) - \varepsilon \|\zeta - \zeta_0\|^2$ will still be plurisubharmonic for all sufficiently small $\varepsilon$. More precisely, there is a neighbourhood $V \subset \mathcal{M}$ of $\zeta_0$, an injective holomorphic mapping $\Phi \colon V \to \mathbb{C}^n$ which sends the point $\zeta_0$ to the origine and a number $\varepsilon_0 >0$ such that the function $\varphi(\zeta) - \varepsilon \|\Phi(\zeta)\|^2$ is plurisubharmonic on $V$ for all $0 \leq \varepsilon < \varepsilon_0$. Since, by Theorem \ref{K-star}, the set $\mathfrak{n}(\mathcal{K})$ is $1$-pseudoconcave, it follows that if we choose $V$ small enough, then the property (3) of Proposition \ref{thm_lmp} with $\mathfrak{n}(\mathcal{K})$ on the place of $A$ holds true. Applying this property to the function $\varphi(\zeta) - \varepsilon \|\Phi(\zeta)\|^2$ with  some $0 < \varepsilon < \varepsilon_0$ and $B$ being the closure of a small enough neighbourhood of $\zeta_0$, we get, by the choice of $\zeta_0$, that 

\begin{equation*} \begin{split} 
\varphi(\zeta_0) &= \varphi(\zeta_0) - \varepsilon \|\Phi(\zeta_0)\|^2 = \max_{\zeta \in {\mathfrak{n}(\mathcal{K})} \cap B} \varphi(\zeta) - \varepsilon \|\Phi(\zeta_0)\|^2 =\max_{\zeta \in {\mathfrak{n}(\mathcal{K})} \cap B} (\varphi(\zeta) - \varepsilon \|\Phi(\zeta)\|^2) \\ & \leq  \max_{\zeta \in {\mathfrak{n}(\mathcal{K})} \cap bB} (\varphi(\zeta) - \varepsilon \|\Phi(\zeta)\|^2)  < \max_{\zeta \in {\mathfrak{n}(\mathcal{K})} \cap bB} \varphi(\zeta) \leq \max_{\zeta \in {\mathfrak{n}(\mathcal{K})} \cap B} \varphi(\zeta) = \varphi(\zeta_0). 
\end{split} \end{equation*}

\noindent
This gives the desired contradiction in the Case 1.

\vspace{2mm}
\noindent{{\bf Case 2.} \sl $\mathfrak{n}(\mathcal{K}) = \emptyset$.}

In this case we prove that the set $\mathcal K$ possesses a smooth strictly plurisubharmonic function, i.e., there is a neighbourhood $\mathfrak{A}$ of $\mathcal K$ in $\mathcal M$ and a smooth strictly plurisubharmonic function $\varphi$ defined on $\mathfrak{A}$. Indeed, since $\bigcap_{{\mathcal K}_{\alpha} \in {\mathcal F}_{\mathcal K}}{\mathcal K}_{\alpha} = \mathfrak{n}(\mathcal{K}) = \emptyset$, and since all the sets ${\mathcal K}_{\alpha} \in {\mathcal F}_{\mathcal K}$ are compact, one has finitely many compacts ${\mathcal K}_{{\alpha}_1}, {\mathcal K}_{{\alpha}_2}, \cdots , {\mathcal K}_{{\alpha}_m}$ in ${\mathcal F}_{\mathcal K}$ such that $\bigcap^m_{j=1}{\mathcal K}_{{\alpha}_j} = \emptyset$. Hence, in view of Lemma \ref{intersection}, there is a finite sequence of spherical cuts ${\mathcal K}_1 \supset {\mathcal K}_2 \supset \cdots \supset {\mathcal K}_m$ such that ${\mathcal K}_1 = \mathcal K$ and ${\mathcal K}_m = \emptyset$. We use this sequence to construct inductively for each $j = m -1, m - 2, \cdots , 1$ a neighbourhood $\mathfrak{A}_j$ of the compact set ${\mathcal K}_j$ and a smooth strictly plurisubharmonic function $\varphi_j$ defined on the set $\mathfrak{A}_j$. If we will be able to perform this construction, then, in view of the equality $\mathcal K = {\mathcal K}_1$, the function $\varphi := \varphi_1$ defined on $\mathfrak{A} := \mathfrak{A}_1$ will be a function as desired.

In order to make the first step of our construction we consider the set ${\mathcal K}_{m-1}$ and observe that, since ${\mathcal K}_m = \emptyset$, and since the set ${\mathcal K}_m$ is obtained from the set ${\mathcal K}_{m-1}$ by a spherical cut, there exist $r_{m-1} \in (0 ,1)$, a neighbourhood $U_{m-1} := U({\hat {\mathbb S}}^n_{r_{m-1}}) \subset {\mathbb C}^n$ of the filled spherical hat ${\hat {\mathbb S}}^n_{r_{m-1}}$ and an injective holomorphic mapping $\Phi_{m-1} \colon U_{m-1} \to \mathcal{M}$ such that $\Phi_{m-1}({\mathbb S}^n_{r_{m-1}}) \subset \mathcal M \setminus {\mathcal K}_{m-1}$ and ${\mathcal K}_{m-1} \subset \Phi_{m-1}({\rm Int}({\hat {\mathbb S}}^n_{r_{m-1}}))$. If we denote by ${\tilde \varphi}_{m-1}$ the restriction to the set $U_{m-1}$ of the function $\varphi_{r_{m-1}}$ provided by Corollary \ref{hat_psh} with $r = r_{m-1}$, then the function $\varphi_{m-1} := {\tilde \varphi}_{m-1} \circ \Phi^{- 1}_{m-1}$ will be smooth and strictly plurisubharmonic on the neighbourhood $\mathfrak{A}_{m - 1} := \Phi_{m-1}({\rm Int}({\hat {\mathbb S}}^n_{r_{m-1}}))$ of the set ${\mathcal K}_{m-1}$.

Now we proceed with the inductive step of our construction. Assume that we have already constructed a neighbourhood $\mathfrak{A}_j$ of the compact set ${\mathcal K}_j$ and a smooth strictly plurisubharmonic function $\varphi_j$ defined on the set $\mathfrak{A}_j$.  After shrinking the set $\mathfrak{A}_j$ if necessary, we can assume that $\mathfrak{A}_j$ has a smooth boundary and the function $\varphi_j$ is smooth on ${\overline{\mathfrak{A}}}_j$, hence there is a smooth extension, which we denote by $\varphi'_j$, of $\varphi_j$ to the whole of $\mathcal M$. Further, since the set ${\mathcal K}_j$ is obtained from the set ${\mathcal K}_{j-1}$ by a spherical cut, there exist $r_{j-1} \in (0 ,1)$, a neighbourhood $U_{j-1} := U({\hat {\mathbb S}}^n_{r_{j-1}}) \subset {\mathbb C}^n$ of the filled spherical hat ${\hat {\mathbb S}}^n_{r_{j-1}}$ and an injective holomorphic mapping $\Phi_{j-1} \colon U_{j-1} \to \mathcal{M}$ such that $\Phi_{j-1}({\mathbb S}^n_{r_{j-1}}) \subset \mathcal M \setminus {\mathcal K}_{j-1}$ and ${\mathcal K}_{j-1} \setminus \Phi_{j-1}({\rm Int}({\hat {\mathbb S}}^n_{r_{j-1}})) = {\mathcal K}_j$. If we denote by ${\tilde \varphi}_{j-1}$ the restriction to the set $U_{j-1}$ of the function $\varphi_{r_{j-1}}$ provided by Corollary \ref{hat_psh} with $r = r_{j-1}$, then the function ${\tilde \varphi}''_{j-1} := {\tilde \varphi}_{j-1} \circ \Phi^{- 1}_{j-1}$ will be smooth and strictly plurisubharmonic on the neighbourhood $\Phi_{j-1}({\rm Int}({\hat {\mathbb S}}^n_{r_{j-1}}))$ of the set ${\mathcal K}_{j-1} \setminus \mathfrak{A}_j$. Moreover, if we denote by $\varphi''_{j-1}$ the extension of the function ${\tilde \varphi}''_{j-1}$ by zero to the rest of the domain ${\mathcal W}_{j-1} := \mathcal M \setminus \Phi_{j-1}({\mathbb S}^n_{r_{j-1}})$, then, in view of the construction of the function $\varphi_{r_{j-1}}$ (see Corollary \ref{hat_psh}), the function $\varphi''_{j-1}$ will be smooth and plurisubharmonic on ${\mathcal W}_{j-1}$, equal to zero on ${\mathcal W}_{j-1} \setminus \Phi_{j-1}({\rm Int}({\hat {\mathbb S}}^n_{r_{j-1}}))$ and positive and strictly plurisubharmonic on  $\Phi_{j-1}({\rm Int}({\hat {\mathbb S}}^n_{r_{j-1}}))$. Finally, since ${\mathcal K}_{j-1}$ is a subset of the open set $\mathfrak{A}_j \cup \Phi_{j-1}({\rm Int}({\hat {\mathbb S}}^n_{r_{j-1}}))$, we can choose for the set $\mathfrak{A}_{j-1}$ an open neighbourhood of the compact ${\mathcal K}_{j-1}$ such that $\mathfrak{A}_{j-1} \subset\subset \mathfrak{A}_j \cup \Phi_{j-1}({\rm Int}({\hat {\mathbb S}}^n_{r_{j-1}}))$ and then we can define the function $\varphi_{j-1} := \varphi'_j + C \varphi''_{j-1}$ with $C > 0$ being chosen so large that $\varphi_{j-1}$ is strictly plurisubharmonic not only on the set $\mathfrak{A}_{j-1} \cap \mathfrak{A}_j$, where $\varphi'_j$ is strictly plurisubharmonic and $C \varphi''_{j-1}$ is plurisubharmonic, but also on the set $\mathfrak{A}_{j-1} \setminus \mathfrak{A}_j \subset\subset \Phi_{j-1}({\rm Int}({\hat {\mathbb S}}^n_{r_{j-1}}))$, where $\varphi''_{j-1}$ is strictly plurisubharmonic, and hence, for large enough $C$, $\varphi'_j + C \varphi''_{j-1} = \varphi_{j-1}$ also is. This proves our argument by induction and hence also the existence of a smooth strictly plurisubharmonic functions on $\mathcal K$ in the Case 2. The proof of the Main Theorem is now completed.  $\hfill\Box$

Note, that using the same argument as in the Case 2 above (i.e. taking a finite sequence of spherical cuts ${\mathcal K}_1 \supset {\mathcal K}_2 \supset \cdots \supset {\mathcal K}_m$ and then applying Corollary \ref{hat_psh}) in the situation when $\mathfrak{n}(\mathcal{K}) \neq \emptyset$, and choosing as a starting point of our inductive construction the function which is identically equal to zero in a neighbourhood of $\mathfrak{n}(\mathcal{K})$, we can obtain the following result. 

\begin{theorem} \label{star psh}
Let $\mathcal{K}$ be a compact subset of a complex manifold $\mathcal M$ and $\mathcal{V}$ be a neighbourhood of the set $\mathfrak{n}(\mathcal{K})$  in $\mathcal M$. Then there is a neighbourhood $\mathfrak{A}$ of $\mathcal K$ in $\mathcal M$ and a smooth nonnegative plurisubharmonic function $\varphi$ defined on $\mathfrak{A}$ such that: 
\begin{enumerate}
	\item[(1)] $\varphi$ is positive and strictly plurisubharmonic on $\mathfrak{A} \setminus \bar{\mathcal{V}}$.
	\item[(2)] $\varphi|_{\mathfrak{n}(\mathcal{K})} \equiv 0$.
\end{enumerate}
\end{theorem}

\section{Applications and open questions}

Now we will discuss some applications of our results and their connection to the other topics which naturally divides this section by the content into four subsections.

\vspace{0,1cm}
\noindent {\bf {1.\,\,Kobayashi hyperbolicity.}} For the first application we need the next statement which gives a link between Kobayashi hyperbolicity and the existence of bounded strictly plurisubharmonic functions. It is just a slight reformulation of Theorem 3 on p. 362 in \cite{Sibony81}.

\begin{theorem} \label{Sibony} 
Let $M$ be a complex manifold which has a bounded continuous strictly plurisubharmonic function. Then $M$ is Kobayashi hyperbolic.
\end{theorem}

As a direct consequence of this statement and our Main Theorem we get the following result.

\begin{theorem} \label{Kobayashi} Let $\mathcal{K}$ be a compact subset of a complex manifold $\mathcal{M}$. If $\mathcal{K}$ does not have 1-pseudoconcave subsets, then there is a neighbourhood $\mathfrak{A}$ of $\mathcal{K}$ in $\mathcal{M}$ which is Kobayashi hyperbolic.
\end{theorem}

\noindent {\bf {2.\,\,The core of a compact.}} The next topic, which is naturally connected to the content of this paper, is related to the notion of the {\it core} of a complex manifold. This notion was introduced and systematically studied by Harz-Shcherbina-Tomassini in \cite{HarzShcherbinaTomassini14} - \cite{HarzShcherbinaTomassini19}. Further results on the foliated structure of the core were obtained by Poletsky-Shcherbina in \cite{PoletskyShcherbina19} and S{\l}odkowski in \cite{Slodkowski19}.

Recall first the definition of the core of a  manifold which was given in \cite{HarzShcherbinaTomassini14} - \cite{HarzShcherbinaTomassini17}:

\begin{definition}  \label{core} Let $\mathcal{M}$ be a complex manifold. Then the set
\begin{equation*}\begin{split}
\mathfrak{c}(\mathcal{M}) \coloneqq \big\{&\zeta \in \mathcal{M} : \text{every smooth plurisubharmonic function on $\mathcal{M}$ that is} \\& \text{bounded from above fails to be strictly plurisubharmonic in } \zeta \big\}
\end{split}\end{equation*}
is called the \emph{core} of $\mathcal{M}$.
\end{definition}

In the same vein we can define a notion of the core in the setting of the present paper.

\begin{definition}  \label{K core} Let $\mathcal{K}$ be a compact subset of a complex manifold $\mathcal{M}$. Then the set
	\begin{equation*}\begin{split}
	\mathfrak{c}(\mathcal{K}) \coloneqq \big\{&\zeta \in \mathcal{K} : \text{every function which is smooth and plurisubharmonic on a neighbour-} \\& \text{hood of $\mathcal{K}$ in $\mathcal{M}$ fails to be strictly plurisubharmonic in } \zeta \big\}
	\end{split}\end{equation*}
	is called the \emph{core} of $\mathcal{K}$.
\end{definition}

This definition obviously implies that the set $\mathfrak{c}(\mathcal{K})$ is compact. Since, by Theorem \ref{star psh}, for each point $\zeta \notin \mathfrak{n}(\mathcal{K})$ there is a smooth plurisubharmonic function defined in a neighbourhood of $\mathcal{K}$ which is strictly plurisubharmonic in $\zeta$, the following property holds true:

\begin{theorem} \label{star core}
	Let $\mathcal{K}$ be a compact subset of a complex manifold $\mathcal M$. Then $\mathfrak{c}(\mathcal{K}) \subset \mathfrak{n}(\mathcal{K})$.
\end{theorem}

We do not know if the reverse statement holds true or not:

\vspace{0,3cm}
\noindent
{\bf Question 1.} {Let $\mathcal{K}$ be a compact subset of a complex manifold $\mathcal M$. Is it always true that $\mathfrak{n}(\mathcal{K}) \subset \mathfrak{c}(\mathcal{K})$?}

\vspace{0,1cm}
One of the most important properties of the core $\mathfrak{c}(\mathcal{M})$, proved in \cite{HarzShcherbinaTomassini17}, is its $1$-pseudoconcavity. The other crucial property of $\mathfrak{c}(\mathcal{M})$, proved in \cite{HarzShcherbinaTomassini18} for manifolds of dimension 2  and in \cite{PoletskyShcherbina19} and \cite{Slodkowski19} for the general case, claims that $\mathfrak{c}(\mathcal{M})$ can be decomposed as a disjoint union of pseudoconcave sets such that every smooth bounded plurisubharmonic function on $\mathcal{M}$ is constant on each of these sets. We do not know if similar statements for the core $\mathfrak{c}(\mathcal{K})$ hold true or not:

\vspace{0,3cm}
\noindent
{\bf Question 2.} {Let $\mathcal{K}$ be a compact subset of a complex manifold $\mathcal M$. Is it always true that $\mathfrak{c}(\mathcal{K})$ is $1$-pseudoconcave?}

\vspace{0,3cm}
\noindent
{\bf Question 3.} {Let $\mathcal{K}$ be a compact subset of a complex manifold $\mathcal M$. Is it always true that $\mathfrak{c}(\mathcal{K})$ can be decomposed as a disjoint union of $1$-pseudoconcave sets $\{E_\alpha\}_{\alpha \in \mathcal A}$ such that $\varphi|_{E_\alpha} \equiv const$ for each $\alpha \in \mathcal A$?}

\vspace{0,2cm}
\noindent {\bf {3.\,\,On the structure of the nucleus.}} Part (3) of Proposition \ref{thm_lmp} above suggests that $1$-pseudoconcave sets resemble in a sense complex analytic varieties. For this reason one can ask a question if such sets and, in particular, in view of Theorem \ref{K-star}, the nucleus $\mathfrak{n}(\mathcal{K})$ of a given compact set $\mathcal{K}$ in a complex manifold $\mathcal{M}$, will have some kind of analytic structure. The answer to this question is, in general, negative even if the dimension of $\mathcal{M}$ is equal to $2$. A corresponding example can be obtained if we choose for $\mathcal{M}$ the $2$-dimensional complex projective space ${\mathbb P}^2$ with the projective coordinates $[z:w:\zeta]$ and for $\mathcal{K}$ the compactification by the point $[1:0:0]$ of the Wermer type set $\mathpzc{E} \subset {\mathbb C}^2_{z, w} \subset {\mathbb P}^2_{z, w, \zeta}$ constructed in \cite{HarzShcherbinaTomassini12}. In this case, by the properties of $\mathpzc{E}$ established in Theorem 1.1 of \cite{HarzShcherbinaTomassini12}, the set $\mathcal{K}$ is $1$-pseudoconcave, hence, $\mathfrak{n}(\mathcal{K}) = \mathcal{K}$, and, moreover, it has no analytic subsets of positive dimension. 

Since the answer to the question above is, in general, negative, one can restrict the question to the case when the compact set $\mathcal{K}$ has more structure, for example, to the case when it is a smooth real hypersurface in $\mathcal{M}$. It is not difficult to see that even in this case the nucleus $\mathfrak{n}(\mathcal{K})$ of $\mathcal{K}$ does not need to contain any analytic subvarieties. Indeed, if we choose the set $\mathcal{K}$ to be a $C^2$-small generic perturbation of the surface $\{[z_1:z_2:z_3:z_4] \in {\mathbb P}^3 : |z_1|^2 + |z_2|^2 - |z_3|^2 - |z_4|^2 = 0\}$ in ${\mathbb P}^3$, then the Levi form of $\mathcal{K}$ has the signature $(1, 1)$, hence $\mathcal{K}$ is $1$-pseudoconcave. However, by genericity of the perturbation, we see that $\mathcal{K}$ can not have any holomorphic disc inside, even locally.

Surprisingly, in the case when the dimension of $\mathcal{M}$ is equal to $2$ and $\mathcal{K}$ is a hypersurface (even not necessarily smooth) in $\mathcal{M}$, the set $\mathfrak{n}(\mathcal{K})$ will have a very special structure:

\begin{theorem} \label{Leaves} 
	Let $\mathcal{M}$ be a complex manifold of dimension $2$ and let $\mathcal{K}$ be a continuous hypersurface of the graph type in $\mathcal{M}$. Then locally the set $\mathfrak{n}(\mathcal{K})$ is a disjoint union of holomorphic discs.
\end{theorem}

Here $\mathcal{K}$ {\it is a continuous hypersurface of the graph type in} $\mathcal{M}$ means that for every point $\zeta \in \mathcal{K}$ there is a neighbourhood $\Omega$ in $\mathcal{M}$ and local holomorphic coordinates $(z, w)$ in $\Omega$ such that $\mathcal{K} \cap \Omega = \{(z, w) \in B^3_r(0) \times {\mathbb R}_v : v = h(z, u)\} =: \Gamma_h$ -- the graph of a continuous function $h : B^3_r(0) \to {\mathbb R}_v$, where $B^3_r(0) := \{(z, u) \in {\mathbb C}_z \times {\mathbb R}_u : |z|^2 + u^2 < r^2\}$ and  $w = u + i v$.

\begin{proof}
    Let $\zeta$ be a point of $\mathfrak{n}(\mathcal{K})$. Since, by our assumptions, $\mathcal{K}$ is a continuous hypersurface of the graph type in $\mathcal{M}$, we can choose a neighbourhood $\Omega$ of $\zeta$ as described above. Moreover, without restriction of generality we can assume, maybe after shrinking $\Omega$ if necessary, that in local holomorphic coordinates $(z, w)$ the set $\Omega$ has the form $\Omega = B^3_r(0) \times (- a, a) \subset B^3_r(0) \times {\mathbb R}_v \subset {\mathbb C}^2_{z,w}$ for some $r, a >0$. Since, by Theorem \ref{K-star}, the set $\mathfrak{n}(\mathcal{K})$ is $1$-pseudoconcave, and since, by part (1) of Proposition \ref{thm_lmp}, $1$-pseudoconcavity is a local property, and taking into account that the statement of the theorem also has a local nature, it will be enough to restrict our consideration to the set $\Omega$.
    
    Consider now the domains $\Omega _{-} := \{(z, w) \in \Omega : v < h(z, u)\}$ and $\Omega _{+} := \{(z, w) \in \Omega : v > h(z, u)\}$ and note that the hulls of holomorphy of these domains are single-sheeted and have the form $\Omega _{\ominus} := \{(z, w) \in \Omega : v < h_{-}(z, u)\}$ and $\Omega _{\oplus} := \{(z, w) \in \Omega : v > h_{+}(z, u)\}$, respectively, where $h_{-} \geq h$ is upper semicontinuous and $h_{+} \leq h$ is lower semicontinuous in $B^3_r(0)$ (the proof of this rather elementary fact can be found, for example, in Lemma 1 of \cite{Chirka73}). Since $\mathfrak{n}(\mathcal{K}) \cap \Omega$ is $1$-pseudoconcave, and since the dimension of $\mathcal{M}$ is equal to $2$, we conclude that the domain $\mathcal W := \Omega \setminus \mathfrak{n}(\mathcal{K})$ is pseudoconvex. The inclusion $\mathfrak{n}(\mathcal{K}) \subset \mathcal{K}$ implies that $\Omega _{-} \subset \mathcal W$ and $\Omega _{+} \subset \mathcal W$ and, hence, by pseudoconvexity of $\mathcal W$, that $\Omega _{\ominus} \subset \mathcal W$ and $\Omega _{\oplus} \subset \mathcal W$. Therefore, one also has the inclusion 

    \begin{equation}\label{eq 1}
    \mathfrak{n}(\mathcal{K}) \cap \Omega \subset b\Omega _{\ominus} \cap b\Omega _{\oplus} \cap \Omega =: E.
    \end{equation}

    \noindent
    The results of \cite{Shcherbina93} (see also Theorem 1 in \cite{Chirka01}) tell us now that the set $E = \bigcupdot_{\alpha \in \mathcal A} E_{\alpha}$ is the disjoint union of complex analytic discs $\{E_{\alpha}\}_{\alpha \in \mathcal A}$ which are closed in $\Omega$. Hence, to finish the proof of the theorem it is enough to prove the following statement:
    
    \vspace{0,2cm}
    \noindent
    {\bf Claim.} {\it If for some $\alpha_0 \in \mathcal A$ one has that $E_{\alpha_0} \setminus \mathfrak{n}(\mathcal{K}) \neq \emptyset$, then $E_{\alpha_0} \cap \mathfrak{n}(\mathcal{K}) = \emptyset$.}

    \vspace{0,2cm}
    \noindent
    This Claim, in view of inclusion \eqref{eq 1}, will imply that $\mathfrak{n}(\mathcal{K}) \cap \Omega = \bigcupdot_{\alpha \in \mathcal B} E_{\alpha}$ for some subset $\mathcal B$ of $\mathcal A$ and, hence, will prove the theorem.

    \vspace{0,2cm}
    \noindent
    {\bf Proof of the Claim.} If $E_{\alpha_0} \setminus \mathfrak{n}(\mathcal{K}) \neq \emptyset$, then we can take a continuous function $h^\ast : B^3_r(0) \to {\mathbb R}_v$ such that $h^\ast \geq h$ on $B^3_r(0)$, $h^\ast = h$ on $\pi_{z,u}(\mathfrak{n}(\mathcal{K}) \cap \Omega)$ and $h^\ast > h$ on $\pi_{z,u}(E_{\alpha_0} \setminus \mathfrak{n}(\mathcal{K}))$, where $\pi_{z,u} : {\mathbb C}^2_{z,w} \to {\mathbb C}_z \times {\mathbb R}_u$ is the canonical projection, and then, as above, we consider the domains $\Omega^\ast _{-} := \{(z, w) \in \Omega : v < h^\ast(z, u)\}$ and $\Omega^\ast _{+} := \{(z, w) \in \Omega : v > h^\ast(z, u)\}$ and their hulls of holomorphy which have the form $\Omega^\ast_{\ominus} := \{(z, w) \in \Omega : v < h^\ast_{-}(z, u)\}$ and $\Omega^\ast_{\oplus} := \{(z, w) \in \Omega : v > h^\ast_{+}(z, u)\}$, respectively, where $h^\ast_{-} \geq h^\ast$ is upper semicontinuous and $h^\ast_{+} \leq h^\ast$ is lower semicontinuous in $B^3_r(0)$. Since $h^\ast = h$ on  $\pi_{z,u}(\mathfrak{n}(\mathcal{K}) \cap \Omega)$, we have that $\mathfrak{n}(\mathcal{K}) \cap \Omega  \subset \Gamma_{h^\ast}$, and then, by pseudoconcavity of $\mathfrak{n}(\mathcal{K})$, we can conclude again from the results of \cite{Shcherbina93} that

    \[
    \mathfrak{n}(\mathcal{K}) \cap \Omega \subset b\Omega^\ast _{\ominus} \cap b\Omega^\ast _{\oplus} \cap \Omega =: E^\ast,
    \]
    where $E^\ast = \bigcupdot_{\alpha \in {\mathcal A}^\ast} E^\ast_{\alpha}$ is the disjoint union of complex analytic discs which are closed in $\Omega$. Now, the property $h^\ast > h$ on $\pi_{z,u}(E_{\alpha_0} \setminus \mathfrak{n}(\mathcal{K})) \neq \emptyset$ tells us that the disc $E_{\alpha_0}$ does not belong to the family $\{E^\ast_{\alpha}\}_{\alpha \in {\mathcal A}^\ast}$, which implies that $E_{\alpha_0} \subset \Omega^\ast_{\ominus}  \subset \mathcal W$ and, hence, also that $E_{\alpha_0} \cap \mathfrak{n}(\mathcal{K}) = \emptyset$. This proves the Claim and completes the proof of Theorem \ref{Leaves}. 
\end{proof}

\vspace{0,2cm}
\noindent 
{\bf Remark.}
	Note, that in the case when the dimension of $\mathcal{M}$ is larger than $2$, a statement analogous to the Theorem \ref{Leaves} is not known and very difficult to get even for $\mathcal{K}$ being the boundary of a smooth pseudoconvex domain $\mathcal D \subset \mathcal{M}$. The problems here are related, in particular, to the jump of the rank of the Levi form and absence of a version of Frobenius theorem for distributions of varying dimension. This kind of difficulties is also present in the following, slightly reformulated here, old problem of Rossi \cite[Conjecture 5.12 on p. 489]{Rossi61} which is, to the best of our knowledge, still open:

\noindent
{\bf Conjecture.}	
	{\it Let $\mathcal D$ be a pseudoconvex domain with smooth boundary in a complex manifold $\mathcal{M}$ of dimension at least $3$. Let $\mathfrak{B}$ be the set of all weakly pseudoconvex points in $b{\mathcal D}$  and ${\rm Int}({\mathfrak{B}})$ is the interior of $\mathfrak{B}$ in $b{\mathcal D}$. Then for each point $\zeta \in {\rm Int}({\mathfrak{B}})$ there is a variety $\mathcal V \subset b{\mathcal D}$ of dimension at least one passing through the point $\zeta$.}

\vspace{3mm}
\noindent

%
%
 \vspace{1truecm}

%
%
%
{\sc N. Shcherbina: Department of Mathematics, University of Wuppertal --- 42119 Wuppertal, Germany}
  
{\em e-mail address}: {\texttt shcherbina@math.uni-wuppertal.de}

\end{document}